\documentclass{amsart}[12pt]
\def\doctype{}

\usepackage{latexsym,amssymb}
\usepackage{color}
\usepackage{fancyhdr}
\usepackage{hyperref}

\newcommand{\colred}{\color[rgb]{.8,0,0}}

\newcommand{\exdiv}{\parallel}

\newcommand{\cB}{\mathcal{B}}

\newcommand\Z{\mathbb{Z}}

\newcommand{\comment}[1]{}

\numberwithin{equation}{section}

\fancypagestyle{titlepage}{
\fancyhead[R]{\doctype}
\fancyhead[CL]{}
\cfoot{\vspace{5pt} \thepage}
}

\let\oldsection\section
\newcommand\boldsection[1]{\oldsection{\bf #1}}
\newcommand\starsection[1]{\oldsection*{\bf #1}}
\makeatletter
\renewcommand\section{\@ifstar\starsection\boldsection}
\makeatother

\newtheoremstyle{theorem}
  {12pt}		  
  {0pt}  
  {\sl}  
  {\parindent}     
  {\bf}  
  {. }    
  { }    
  {}     
\theoremstyle{theorem}
\newtheorem{thm}{Theorem}[section]  
\newtheorem{lemma}[thm]{Lemma}     

\newtheorem{conj}[thm]{Conjecture}

\newtheorem{prop}[thm]{Proposition}

\newtheoremstyle{definition}
  {12pt}		  
  {0pt}  
  {}  
  {\parindent}     
  {\bf}  
  {. }    
  { }    
  {}     
\theoremstyle{definition}

\newcommand\rk{{\sc Remark.} }

\renewcommand{\proofname}{Proof}

\makeatletter
\renewenvironment{proof}[1][\proofname]{\par
  \pushQED{\qed}%
  \normalfont \partopsep=\z@skip \topsep=\z@skip
  \trivlist
  \item[\hskip\labelsep
        \scshape
    #1\@addpunct{.}]\ignorespaces
}{%
  \popQED\endtrivlist\@endpefalse
}
\makeatother

\makeatletter
\renewcommand*\@maketitle{%
  \normalfont\normalsize
  \@adminfootnotes
  \@mkboth{\@nx\shortauthors}{\@nx\shorttitle}%
  \global\topskip42\p@\relax 
  \@settitle
  \ifx\@empty\authors \else {\vskip 1em
\vtop{\centering\shortauthors\@@par}} \fi
  \ifx\@empty\@date \else {\vskip 1em \vtop{\centering\@date\@@par}}\fi 
  \ifx\@empty\@dedicatory
  \else
    \baselineskip18\p@
    \vtop{\centering{\footnotesize\itshape\@dedicatory\@@par}%
      \global\dimen@i\prevdepth}\prevdepth\dimen@i
  \fi
  \@setabstract
  \normalsize
  \if@titlepage
    \newpage
  \else
    \dimen@34\p@ \advance\dimen@-\baselineskip
    \vskip\dimen@\relax
  \fi
} 
\renewcommand*\@adminfootnotes{%
  \let\@makefnmark\relax  \let\@thefnmark\relax
  \ifx\@empty\@subjclass\else \@footnotetext{\@setsubjclass}\fi
  \ifx\@empty\@keywords\else \@footnotetext{\@setkeywords}\fi
  \ifx\@empty\thankses\else \@footnotetext{%
    \def\par{\let\par\@par}\@setthanks}%
  \fi
\thispagestyle{titlepage}
}
\makeatother

\begin{document}

\title[MOLS with large holes]{\large Mutually orthogonal latin squares with large holes}

\author{Peter J.~Dukes}
\address{
Mathematics and Statistics,
University of Victoria, Victoria, Canada
}
\email{dukes@uvic.ca, cvanbomm@uvic.ca}

\author{Christopher M. van Bommel}

\thanks{Research of the first author is supported by NSERC grant number 312595--2010;
research of the second author is supported by a CGS-M scholarship}

\date{\today}

\begin{abstract}
Two latin squares are orthogonal if, when they are superimposed,
every ordered pair of symbols appears exactly once.  This definition extends
naturally to `incomplete' latin squares each having a hole on the same rows,
columns, and symbols.  If an incomplete latin square of order $n$ has a hole
of order $m$, then it is an easy observation that $n \ge 2m$.  More
generally, if a set of $t$ incomplete mutually orthogonal latin squares of
order $n$ have a common hole of order $m$, then $n \ge (t+1)m$.  In this
article, we prove such sets of incomplete squares exist for all $n,m \gg 0$
satisfying $n \ge 8(t+1)^2 m$.
\end{abstract}

\maketitle
\hrule



\section{Introduction}

A \emph{latin square} is an $n \times n$ array with entries from an
$n$-element set of symbols such that every row and column is a permutation
of the symbols.  Often the symbols are taken to be from
$[n]:=\{1,\dots,n\}$.  The integer $n$ is called the \emph{order} of the
square.

Two latin squares $L$ and $L'$ of order $n$ are \emph{orthogonal} if
$\{(L_{ij},L'_{ij}): i,j \in [n]\}=[n]^2$; that is, two
squares are orthogonal if, when superimposed, all ordered pairs of symbols
are distinct.  The following arrangement of playing cards illustrates a pair
of orthogonal latin squares of order 4.
\begin{center}
\begin{tabular}{|cccc|}
\hline
A$\spadesuit$ & \colred J$\heartsuit$ & Q$\clubsuit$ & \colred
K$\diamondsuit$ \\
J$\clubsuit$ & \colred A$\diamondsuit$ & K$\spadesuit$ & \colred
Q$\heartsuit$ \\
\colred Q$\diamondsuit$ & K$\clubsuit$ & \colred A$\heartsuit$ &
J$\spadesuit$ \\
\colred K$\heartsuit$ & Q$\spadesuit$ & \colred J$\diamondsuit$ &
A$\clubsuit$ \\
\hline
\end{tabular}
\end{center}
Euler's famous `36 officers problem' asks whether there exists a pair of
orthogonal latin squares of order six.  The answer in that case is negative.

A family of latin squares in which any pair are orthogonal is 
called a set of \emph{mutually orthogonal latin squares}, or `MOLS' for
short.  The maximum size of a set of MOLS of order $n$ is denoted $N(n)$.
It is easy to see that $N(n) \le n-1$ for $n>1$,  with equality if and only
if there exists a projective plane of order $n$.  Consequently, $N(q)=q-1$
for prime powers $q$.  Using number sieves and some `gluing' constructions,
it has been shown in \cite{Beth} (building upon \cite{CES,WilsonMOLS}) that
$N(n) \ge n^{1/14.8}$ for large $n$.

In this article, we are interested in an `incomplete' variant on MOLS.
First, an \emph{incomplete latin square} of order $n$ with a \emph{hole} of
order $m$ is an $n \times n$ array $L=(L_{ij}: i,j \in [n])$ on $n$ symbols
(let us say $[n]$ for convenience) together with a \emph{hole} $M \subseteq
[n]$ such that
\begin{itemize}
\item
$L_{ij}$ is empty if $\{i,j\} \subseteq M$;
\item
$L_{ij}$ contains exactly one symbol if $\{i,j\} \not\subseteq M$;
\item
every row and every column in $L$ contains each symbol at most once; and
\item
symbols in $M$ do not appear in rows or columns indexed by $M$.
\end{itemize}
Note that $M$ is often taken to be an interval of consecutive
rows/columns/symbols (but need not be).  The definition is meant to extend
to any set of symbols (and corresponding hole symbols).  One feature of
incomplete latin squares is that they can `frame' latin subsquares on the
hole.  An example in the case $n=5$, $m=2$ is shown below.
\begin{center}
\begin{tabular}{|ccccc|}
\hline
 &  & 3 & 4 & 5 \\
 &  & 4 & 5 & 3 \\
3 & 4 & 1 & 2 & 5 \\
4 & 5 & 2 & 3 & 1 \\
5 & 3 & 4 & 1 & 2 \\
\hline
\end{tabular}
\end{center}

Two incomplete latin squares $L,L'$ on $[n]$ with common holes $M$ are
orthogonal if
$$\{(L_{ij},L'_{ij}): \{i,j\} \not\subseteq M \}=[n]^2 \setminus M^2,$$
and as before we can have sets of mutually orthogonal incomplete latin
squares.  A set of $t$ such incomplete squares of order $n$ with holes of
order $m$ is denoted $t$-IMOLS$(n;m)$.  Note that the case $m=0$ or $1$
reduces to ordinary MOLS.  For one noteworthy example, there exist
2-IMOLS$(6;2)$ (see \cite{Handbook} for instance) despite the nonexistence
of orthogonal latin squares of order 6 (or 2).

It is a straightforward counting argument (see \cite{Horton}) that the
existence of $t$-IMOLS$(n;m)$ requires
\begin{equation}
\label{ineq}
n \ge (t+1)m.
\end{equation}
The special case $t=1$ recovers the familiar condition that latin subsquares
cannot exceed half the size of their embedding.  On the other hand, $n\ge
2m$ is sufficient for the existence of an incomplete latin square of order
$n$ with a hole of order $m$.   For $t=2,3$, the inequality (\ref{ineq}) is
known to be sufficient, except for small cases; see \cite{HZ,AD}.  The best
presently known result for $t=4$ is a by-product of work on 6-IMOLS$(n;m)$
in \cite{CZ}, and so in this case $n \sim 7m$ is a barrier.  This gives some
evidence of the difficulty of constructing $t$-IMOLS$(n;m)$ near the bound
for general $t$.

For our main result, we prove sufficiency for large $n$ and $m$ when
(\ref{ineq}) is strengthened a little.

\begin{thm}
\label{main}
There exist $t$-IMOLS$(n;m)$ for all sufficiently large $n,m$ satisfying $n
\ge 8(t+1)^2m$.
\end{thm}

We actually obtain Theorem~\ref{main} as a consequence of a more general
result on pairwise balanced block designs with holes.  The corresponding
`inequality' we obtain in this more general case is probably far from best
possible, but it reduces to a reasonable condition for our application to
IMOLS.  The next three sections of the article are devoted to the
development and proof of our result on designs with holes.  In Section 5, we
conclude with a proof of Theorem~\ref{main} and discussion of a few related
items.

\section{Background on block designs}

Let $v$ be a positive integer and $K \subseteq \Z_{\ge 2} :=
\{2,3,4,\dots\}$.  A \emph{pairwise balanced design}
PBD$(v,K)$ is a pair $(V,\cB)$, where
\begin{itemize}
\item
$V$ is a $v$-element set of \emph{points};
\item
$\cB \subseteq \cup_{k \in K} \binom{V}{k}$ is a family of of subsets of
$V$, called \emph{blocks}; and
\item
every two distinct points appear together in exactly one block.
\end{itemize}

In a PBD$(v,K)$, the pairs covered by each block must partition
$\binom{V}{2}$.
In addition, for any point $x \in V$, the remaining $v-1$ points must
partition into `neighborhoods' in the blocks incident with $x$.  It is
helpful to think of the resulting divisibility restrictions as,
respectively, `global' and `local' conditions, which we state below (in
reverse order).

\begin{prop}
\label{neccond}
The existence of a PBD$(v,K)$ implies
\begin{eqnarray}
\label{local-pbd}
v-1 &\equiv& 0 \pmod{\alpha(K)} ~\text{and}\\
\label{global-pbd}
v(v-1)  &\equiv& 0 \pmod{\beta(K)},
\end{eqnarray}
where $\alpha(K):=\gcd\{k-1:k \in K\}$ and $\beta(K):=\gcd\{k(k-1):k \in
K\}$.
\end{prop}

The sufficiency of these conditions for $v \gg 0$ is a celebrated result due
to Richard M.~Wilson.

\begin{thm}[Wilson, \cite{RMW1}]
\label{asym-pbd}
Given any $K \subseteq \Z_{\ge 2}$, there exist PBD$(v,K)$ for all
sufficiently large $v$ satisfying  $(\ref{local-pbd})$ and
$(\ref{global-pbd})$.
\end{thm}

Theorem~\ref{asym-pbd} is the foundation for an existence theory of many
`PBD-closed' combinatorial structures, including graph decompositions.
See \cite{LW} for a very general extension and good survey of the
applications.  Even the recent theorem of Keevash \cite{Keevash} on
$t$-designs extends a few aspects of Wilson's proof of Theorem~\ref{asym-pbd}.

Now, let $v \ge w$ be positive integers and $K\subseteq \Z_{\ge 2}$.  An
\emph{incomplete pairwise balanced design} IPBD$((v;w),K)$ is a triple
$(V,W,\cB)$ such that
\begin{itemize}
\item
$V$ is a set of $v$ points and $W \subset V$ is a \emph{hole} of size $w$;
\item
$\cB \subseteq \cup_{k \in K} \binom{V}{k}$ is a family of blocks;
\item
no two distinct points of $W$ appear in a block; and
\item
every two distinct points not both in $W$ appear together in exactly one
block.
\end{itemize}

A closely related notion is that of a PBD$(v,K)$ containing a `subdesign' or
`flat' PBD$(w,K)$.  When such a subdesign exists, we obtain an
IPBD$((v;w),K)$ by taking the difference of block sets.  Likewise, an IPBD
with hole $W$ can be `filled' with a PBD (or another IPBD) on $W$, but only
when this smaller design exists.

The case $w=v$ leads to $\cB=\emptyset$ and we exclude this in what follows.
The case $w=1$ reduces to a PBD$(v,K)$, since such a hole contains no pairs.

By analogy with (\ref{local-pbd}) and (\ref{global-pbd}), there are naive
divisibility conditions on the parameters.

\begin{prop}
\label{neccond-h}
The existence of an IPBD$((v;w),K)$ implies
\begin{eqnarray}
\label{local}
v-1~\equiv~w-1 &\equiv& 0 \pmod{\alpha(K)}, ~\text{and}\\
\label{global}
v(v-1) - w(w-1) &\equiv& 0 \pmod{\beta(K)}.
\end{eqnarray}
\end{prop}

We say integers $v$ and $w$ are \emph{admissible} (for IPBD with block sizes
$K$) if (\ref{local}) and (\ref{global}) hold.  There is another necessary
condition taking the form of an inequality.

\begin{prop}
\label{holesize-bound}
Let $k:=\min K$.  The existence of an IPBD$((v;w),K)$ with $v>w$ implies
\begin{eqnarray}
\label{ineq-h}
v \ge (k-1)w +1.
\end{eqnarray}
Equality holds if and only if every block intersects the hole and has size exactly $k$.
\end{prop}

\begin{proof}
This is an easy adaptation of the argument in \cite{DLL2}, which handles the case $K=\{k\}$.  
A point outside the hole must appear in: (1) at least $w$ blocks, as no two points in the hole can be in the same block; and (2) at most $\frac{v - 1}{k - 1}$ blocks.  
\end{proof}

The case $v=(k-1)w+1$ is equivalent (upon truncating points from the hole)
to 'resolvable' designs with $v-w$ points and block size $k-1$, where we say
a design $(V,\cB)$ is \emph{resolvable} if $\cB$ can be partitioned into
1-factors, or \emph{parallel classes} on $V$.  It is known that resolvable designs with block
size $k$ exist when possible for $v \gg 0$.

\begin{thm}[\cite{RCW}]
\label{thm:rpbd:asym}
There exists resolvable PBD$(v,\{k\})$ for all sufficiently large $v$
satisfying $v \equiv k \pmod{k(k-1)}$.
\end{thm}

\rk
The above necessary congruence on $v$ comes from $k \mid v$ (for a parallel
class) and (\ref{local-pbd}).

Existence of IPBDs in the case $K=\{k\}$ was recently considered by the
first author, Ling and Lamken.  The result comes very close
to proving sufficiency of the preceding conditions.

\begin{thm}[\cite{DLL2}]
\label{ipbd-k}
Let $k$ be an integer at least two.  For every real number $\epsilon>0$,
there exists IPBD$((v;w),\{k\})$ for all sufficiently large admissible $v,w$
satisfying $(\ref{local})$, $(\ref{global})$, and $v \ge (k-1+\epsilon) w$.
\end{thm}

For the proof of Theorem~\ref{ipbd-k}, and our extension to general $K$
which follows, we use a common generalization of MOLS and PBDs (and their
incomplete variants).  Let $T$ denote an integer partition of $v$.
A \emph{group divisible design} of \emph{type} $T$ with block sizes in $K$,
denoted GDD$(T,K)$, is a triple $(V,\Pi,\cB)$ such that
\begin{itemize}
\item
$V$ is a set of $v$ points;
\item
$\Pi=\{V_1,\dots,V_u\}$ is a partition of $V$ into \emph{groups} so that
$T=\{|V_1|,\dots,|V_u|\}$;
\item
$\cB \subseteq \cup_{k \in K} \binom{V}{k}$ is a set of of blocks meeting
each group in at most one point; and
\item
any two points from distinct groups $V_j$ appear together in exactly one
block.
\end{itemize}

It is common to use `exponential notation' such as $g^u$ to stand for $u$
groups of size $g$.  For instance, a \emph{transversal design} TD$(k,n)$ is
a GDD$(n^k,\{k\})$.  In this case, the blocks are transversals of the
partition.  A TD$(k,n)$ is equivalent to $k-2$ MOLS of order $n$, where two
groups are reserved to index the rows and columns of the squares.

In general, a group divisible design of type $T=g^u$ is called
\emph{uniform}. There is a satisfactory asymptotic existence result for such
objects, stated here for later use.

\begin{thm}[Draganova \cite{Draganova} and Liu \cite{Liu}]
\label{asym-gdd}
Given $g$ and $K \subseteq \Z_{\ge 2}$, there exists a GDD$(g^u,K)$ for all
sufficiently large $u$ satisfying
\begin{eqnarray}
\label{local-gdd}
g(u-1) &\equiv& 0 \pmod{\alpha(K)} ~\text{and}\\
\label{global-gdd}
g^2 u(u-1) &\equiv& 0 \pmod{\beta(K)}.
\end{eqnarray}
\end{thm}

A nice feature of GDDs is that their groups act as a partition into holes;
each can be `filled' with PBDs (or smaller GDDs).  Another feature of GDDs
is that they admit a natural `expansion' of their groups.  This is made
precise in the next construction.  The idea is simply to independently
replicate the points of a `master' GDD, replacing its blocks by `ingredient'
GDDs of the right type.

\begin{lemma}[Wilson's fundamental construction, \cite{ConsUses}]
\label{wfc-lemma}
Suppose there exists a GDD $(V,\Pi,\cB)$, where $\Pi=\{V_1,\dots,V_u\}$.
Let $\omega:V \rightarrow \Z_{\ge 0}$, assigning nonnegative weights to each
point in such a way that for every $B \in \cB$ there exists a
GDD$([\omega(x) : x \in B],K)$.  Then there exists a GDD$(T,K)$, where
$$T=\left[\sum_{x \in V_1} \omega(x),\dots,\sum_{x \in V_u}
\omega(x)\right].$$
\end{lemma}

An \emph{incomplete group divisible design}, or IGDD, is a quadruple
$(V,\Pi,\Xi,\cB)$ such that $V$ is a set of $v$ points, $\Pi =
\{V_1,\dots,V_u)$ is a partition of $V$ into `groups',
$\Xi=\{W_1,\dots,W_u\}$ with $W_i \subseteq V_i$ called `holes' for each
$i$, and $\cB$ is a set of blocks (say with sizes in $K$) such that
\begin{itemize}
\item
two points get covered by a block (exactly one block) if and only if they
come from different groups, say $V_i$ and $V_j$, $i \neq j$, {\bf and} are not both in to the corresponding holes $W_i$ and $W_j$.
\end{itemize}

As with GDDs, the type of an IGDD can be written by listing, using
exponential notation when appropriate, the pairs $(|V_i|;|W_i|)$ of group
size and corresponding hole size.  So, for example, a (uniform)
IGDD in which every group $V_i$ has size
$g$ and every hole $W_i$ has size $h$ is denoted by IGDD$((g;h)^u,K)$.

The necessary divisibility conditions follow a similar structure as before.

\begin{prop}
\label{neccond-igdd}
The existence of an IGDD$((g;h)^u,K)$ implies
\begin{eqnarray}
\label{local-igdd}
g(u-1)~\equiv~h(u-1) &\equiv& 0 \pmod{\alpha(K)},  ~\text{and} \\
\label{global-igdd}
(g^2-h^2)u(u-1) &\equiv& 0 \pmod{\beta(K)}.
\end{eqnarray}
\end{prop}

Also, by analogy with Proposition~\ref{holesize-bound} we have
\begin{equation}
\label{ineq-igdd}
g \ge (k-1)h,
\end{equation}
where again $k=\min K$.

From the method of edge-colored graph decompositions in \cite{LW}, it is
possible to get a `large $u$'  existence theory for uniform IGDDs.  The
proof is a straightforward extension of the argument used to prove
Theorem~\ref{asym-gdd}; details can be found in \cite{CvB}.

\begin{thm}[\cite{CvB}]
\label{asym-igdd}
Let $k = \min K$.  Given integers $g,h$ with $g \ge (k-1)h$, there exists an
IGDD$((g;h)^u,K)$ whenever $u$ is sufficiently large satisfying {\rm
(\ref{local-igdd})} and {\rm (\ref{global-igdd})}.
\end{thm}

The following constructions produce IPBDs from GDDs and IGDDs.  In the first
case, all but one group of a GDD gets filled, and in the second case, the
holes in an IGDD get merged into a larger hole.

\begin{lemma} \label{const:gdd-ipbd}
Suppose there exists a GDD$(T,K)$ on $v$ points and let one of its groups
have size $a$.  If, for all other group sizes $g$ in $T$, there exists an
IPBD$((g+i; i),K)$, then there exists an IPBD$((v+i;a+i);K)$.
\end{lemma}

\begin{lemma} \label{const:igdd-ipbd}
Suppose there exists an IGDD$(T, K)$ on $v$ points with $w$ hole points in
total.  If, for each $(g; h)$ in $T$, there exists an IPBD$((g + i; h + i),
K)$, then there exists an IPBD$((v + i; w + i), K)$.
\end{lemma}

The foregoing designs and constructions are enough to prove a preliminary
common generalization of Theorems~\ref{asym-pbd} and \ref{ipbd-k}.  This is
our topic in the next two sections.

\section{Incomplete designs with fixed hole size}

Our purpose here is to prove the following `large $v$' result for
IPBD$((v;w),K)$.
This is the first major step toward the `two parameter' result we desire.

\begin{thm} \label{thm:ipbd:K:fixed-w}
Given $w \equiv 1 \pmod{\alpha(K)}$, there exist IPBD$((v; w), K)$ for all
sufficiently large $v$ satisfying $(\ref{local})$ and $(\ref{global})$.
\end{thm}

For the proof, we use IGDDs along with two different classes of IPBDs.  The
first of these classes has $v \equiv w$ modulo a large multiple of
$\beta(K)$, and its proof is an easy application of `nearly uniform' GDDs.
In what follows, for convenience we write $\alpha:=\alpha(K)$, $\beta:=\beta(K)$, and $\gamma:=\beta/\alpha$.  It is easy to see that
$\gamma$ is an integer coprime with $\alpha$.

\begin{lemma}[See \cite{DLL2,CvB}] \label{lem:nugdd:asym}
For any $m \gg 0$ with $m \equiv 0 \pmod{\gamma}$ and any $t \equiv 0
\pmod{\alpha}$, there exists a GDD$(s^m t^1, K)$ for all sufficiently
large integers $s \equiv 0 \pmod{\alpha}$.
\end{lemma}

Now let $M := m \beta$, where $m$ satisfies the conditions of
Lemma~\ref{lem:nugdd:asym}.  Our first class of IPBDs results from filling
groups; the small group becomes the hole and the other groups are filled
with PBDs.

\begin{prop} \label{prop:ipbd:class1}
For any $w \equiv 1 \pmod{\alpha}$, there exist IPBD$((v; w), K)$ for all
sufficiently large $v \equiv w \pmod{M}$.
\end{prop}

\begin{proof}
Let $v - w = a M = a m \beta$.  We assume $a$ is large enough such that
there exists both a GDD$((a \beta)^m (w - 1)^1, K)$ by
Lemma~\ref{lem:nugdd:asym} and a PBD$(a \beta + 1, K)$ by
Theorem~\ref{asym-pbd}.  Finally, use Lemma~\ref{const:gdd-ipbd} with $i=1$
to obtain an IPBD$((v; w), K)$.
\end{proof}

We now construct a second class of incomplete pairwise balanced designs.
Here, the parameters are such that $v \equiv 1 - w \pmod{\gamma}$.  Our
approach is to start with an appropriate resolvable pairwise balanced design
using a single block size, and then fill each of the blocks using block
sizes in $K$.

\begin{prop} \label{prop:ipbd:class2}
Given $K$, a modulus $M = m \beta$, and an admissible congruence
class $w_0 \pmod{M}$ for incomplete pairwise balanced designs with block
sizes in $K$, there exists an IPBD$((v; w_1), K)$ with $w_1 \equiv w_0
\pmod{M}$ and $v \equiv 1 - w_1 \pmod{\gamma}$.
\end{prop}

\begin{proof}
Choose an integer $q \gg 0$ such that $\gcd(q, M) = 1$, $q \alpha + 1
\equiv 0 \pmod{\gamma}$, and there exists a PBD$(q \alpha + 1, K)$,
whose existence follows from Theorem~\ref{asym-pbd}.  Since $q$ and $M$ are
coprime, $q \alpha$ and $M$ have only the common factor $\alpha$, and
hence it follows from the Chinese remainder theorem that we can choose a
$w_1 \gg 0$ such that $w_1 \equiv w_0 \pmod{M}$ and $w_1 \equiv 1 \pmod{q
\alpha}$ and such that there exists a resolvable PBD$(w_1 (q \alpha -
1) + 1, q \alpha)$ by Theorem~\ref{thm:rpbd:asym}.  It follows that there
exists an IPBD$((w_1 q \alpha + 1; w_1), q \alpha + 1)$. 
Replacing each block with a PBD$(q \alpha+1,K)$ 
results in an IPBD$((w_1 q \alpha + 1; w_1), K)$ with $v = w_1
q \alpha + 1 \equiv 1 - w_1 \pmod{\gamma}$ as required.
\end{proof}

We can now fill groups of IGDDs to find an example incomplete pairwise
balanced design for each admissible congruence class of $v$ and $w$.
The following extends a similar argument in \cite{DLL2}.

\begin{prop} \label{prop:ipbd:class-all}
Given $K$, a positive modulus $M = m \beta$, and admissible congruence
classes $v_0, w_0 \pmod{M}$ for incomplete pairwise balanced designs with
block sizes in $K$, there exists an IPBD$((v_2; w_2), K)$ for some $v_2
\equiv v_0$ and $w_2 \equiv w_0 \pmod{M}$.
\end{prop}

\begin{proof}
The incomplete pairwise balanced designs constructed in
Proposition~\ref{prop:ipbd:class2} can be used as ingredients in
Lemma~\ref{const:igdd-ipbd} to produce the remaining examples outside the
two classes previously considered, however, we will therefore require
certain conditions on $v$ and $w$.  In particular, if $q$ retains its value
from Proposition~\ref{prop:ipbd:class2}, where it is chosen independently of
$w_0$, then we must have $v \equiv w \equiv 1 \pmod{q}$.  Hence, we must
select classes $v_1$ and $w_1 \pmod{Mq}$ such that $v_1 \equiv v_0
\pmod{M}$, $v_1 \equiv 1 \pmod{q}$, $w_1 \equiv w_0 \pmod{M}$, and $w_1
\equiv 1 \pmod{q}$, which can be found by the Chinese remainder theorem as
$\gcd(q, M) = 1$.

Let the incomplete pairwise balanced designs found in
Proposition~\ref{prop:ipbd:class2} be denoted as IPBD$((x; y), K)$, where $x
= y q \alpha + 1$.  If we use the uniform incomplete group divisible
design IGDD$((g; h)^u, K)$, with $g - h = x - y$ and $y \ge h$, then
applying Lemma~\ref{const:igdd-ipbd} results in an IPBD$((g (u - 1) + x; h
(u - 1) + y), K)$.  Eliminating some dependent parameters, we have
\begin{align*}
v_2 - w_2 &= (g - h) (u - 1) + x - y = (x - y) u \\
 &= (y q \alpha + 1 - y) u = (w_2 - h (u - 1)) (q \alpha - 1) u + u.
\end{align*}
Or, rearranging,
\begin{equation}
\label{eq-uh}
u (u - 1) (q \alpha - 1) h = w_2 u (q \alpha - 1) + w_2 + u - v_2.
\end{equation}

Now, we wish to determine $u$ and $h$ such that $v_2 \equiv v_1$ and $w_2
\equiv w_1 \pmod{Mq}$.  Hence, it is sufficient to determine the required
congruence classes for $u$ and $h$.  In view of (\ref{eq-uh}), we solve the congruence
$$ u (u - 1) (q \alpha - 1) h \equiv w_1 u (q \alpha - 1) + w_1 + u -
v_1 \pmod{p^t} $$
for each prime power $p^t$ such that $p^t \exdiv Mq$.  
Now, we choose the solution
\begin{equation}
\label{u-choice}
u \equiv \begin{cases}
- (q \alpha - 1)^{-1} & \mbox{if $\gcd(p, q \alpha - 1) = 1$}, \\
(v_1 - w_1) (w_1 (q \alpha - 1) + 1)^{-1} & \mbox{otherwise.}
\end{cases}
\end{equation}
Since $p$ cannot divide two consecutive values, it follows that both
inverses in (\ref{u-choice}) exist when required.  If $\gcd(p, q \alpha - 1) = 1$, we  obtain 
$(u - 1) h \equiv v_1 - u$, and hence
\begin{equation}
\label{h-choice}
h \equiv -\frac{(q \alpha - 1) v_1 + 1}{q \alpha} \pmod{p^t}.
\end{equation}
Note that the fraction in (\ref{h-choice}) is well defined since $v_1$ admissible implies $v_1 \equiv 1  \pmod{\alpha}$, also
$v_1 \equiv 1 \pmod{q}$ as a result of
Proposition~\ref{prop:ipbd:class2}, and $\gcd(\alpha, q) = 1$ as
$\alpha \mid M$.  In the case $p \mid q \alpha - 1$, note that from (\ref{eq-uh}) and (\ref{u-choice}) we have 
$u(u - 1) (q \alpha - 1) h \equiv 0$, and hence $h \equiv 0$ suffices.  The Chinese remainder theorem gives a simultaneous solution for $u,h$, which we 
summarize in Table~\ref{tbl:ipbd:param}.

\begin{table}[htbp]
\begin{center}
\caption{\label{tbl:ipbd:param} Choice of Parameters to Obtain a Desired
Congruence Class}
\begin{tabular}{c|c|c}
& $\gcd(p, q \alpha - 1) = 1$ & $p \mid q \alpha - 1$ \\
\hline
$u \equiv$ & $-(q \alpha - 1)^{-1}$ & $(v_1 - w_1) ((q \alpha - 1) w_1
+ 1)^{-1}$ \\
$h \equiv$ & $-((q \alpha - 1) v_1 + 1) / q \alpha$ & $0$
\end{tabular}
\begin{align*}
y &= w_1 - h (u - 1) \\
x &= y q \alpha + 1 \\
g &= h + x - y
\end{align*}
\end{center}
\end{table}

We now verify that the required incomplete group divisible design exists by
Theorem~\ref{asym-igdd}.  Checking (\ref{local-igdd}), we have
$$ u - 1 \equiv - (q \alpha - 1)^{-1} - 1 \equiv \frac{-q \alpha}{q \alpha - 1} \equiv 0
\pmod{p^t}$$ for any $p^t \exdiv \alpha$, and so it follows that $u - 1 \equiv 0 \pmod{\alpha}$.

Checking (\ref{global-igdd}), we calculate
\begin{align*}
(g^2 - h^2) u (u - 1) &\equiv (g - h) u (g + h) (u - 1) \\
&\equiv (v_1 - w_1) (v_1 + w_1 - y (q \alpha + 1) - 1) \\
&\equiv (v_1 - w_1) (v_1 + w_1 - 1) \\
 & \equiv 0 \pmod{\gamma}.
\end{align*}

Thus, the required IGDD$((g; h)^u, K)$ exists provided $u$ is sufficiently
large.  Finally, Lemma~\ref{const:igdd-ipbd} results in an IPBD$((v_2;
w_2), K)$ hitting the desired congruence classes.
\end{proof}

We can now prove the desired result on incomplete pairwise
balanced designs with fixed hole size.

\begin{proof}[Proof of Theorem~\ref{thm:ipbd:K:fixed-w}]
Let $v$ be sufficiently large satisfying (\ref{local}) and (\ref{global}).
By Proposition~\ref{prop:ipbd:class-all}, there exists an IPBD$((v_2; w_2),
K)$ such that $v_2 \equiv v$ and $w_2 \equiv w \pmod{M}$.  We can assume $v
\gg v_2$ and $w_2 \gg w$ so that there exist both an IPBD$((v; v_2), K)$ and
an IPBD$((w_2; w), K)$ by Proposition~\ref{prop:ipbd:class1}.  Then an
IPBD$((v; w), K)$ exists as a result of filling holes (twice).
\end{proof}

\section{Incomplete designs with growing hole size}

This section extends the strategy which Dukes, Lamken, and
Ling~\cite{DLL2} used to prove Theorem~\ref{ipbd-k}.  Our goal is a proof of the
following.

\begin{thm}
\label{ipbd}
Let $K_0 \subseteq K$ with $\alpha(K_0)=\alpha(K)$.  There exists
IPBD$((v;w),K)$ for all sufficiently large admissible $v,w$ satisfying
$(\ref{local})$, $(\ref{global})$, and $v \ge (\prod_{k \in K_0} k) w$.
\end{thm}

\rk
We can actually do a little bit better, replacing factors of $k$ with 
$k-1+\epsilon$, where small $\epsilon$ drives up the choice of $v,w$.

We first note two important ingredient designs with a single block size $k$.
For convenience here, we abbreviate `$\{k\}$' simply by `$k$' in the notation.

\begin{lemma}[\cite{DLL2}] \label{lem:ipbd:gdd-max}
For sufficiently large $m$ with $m \equiv -1 \pmod{k}$ and $m \equiv 1
\pmod{k - 2}$, there exist both GDD$((k - 1)^m r^1, k)$ and GDD$((k -
1)^{m + 1} r^1, k)$, where $r = (k - 1) (m - 1) / (k - 2)$.
\end{lemma}

\begin{lemma}[\cite{DLL2}] \label{lem:ipbd:gdd-small}
Let $s$ be an integer with $s \equiv 0 \pmod{k - 1}$ and $s \equiv -1
\pmod{k}$.  There exist both GDD$((k - 1)^m s^1, k)$ and GDD$((k - 1)^{m
+ 1} s^1, k)$ for all sufficiently large $m \equiv -1 \pmod{k}$.
\end{lemma}

Now, we have the following construction, adapted from \cite{DLL2} to handle
$K$ instead of $\{k\}$.

\begin{lemma} \label{lem:ipbd:recur}
For any real $\epsilon > 0$ and a given $k \in K$, there exist IPBD$((v; w),
K)$ for all sufficiently large $v, w$ satisfying $(\ref{local})$,
$(\ref{global})$, $v > (k - 1 + \epsilon) w$, and $v - w \equiv 0 \pmod{k -
1}$.
\end{lemma}

\begin{proof}
Let $m$ be sufficiently large such that for each $x \in R := \{k - 1, k^2 -
1, r\}$, there exist both GDD$((k - 1)^m x^1, k)$ and GDD$((k - 1)^{m + 1}
x^1, k)$.  As $r$ retains its value of $(k - 1) (m - 1) / (k - 2)$ from
Lemma~\ref{lem:ipbd:gdd-max}, $m$ is restricted as stated, and as $k - 1
\equiv k^2 - 1 \equiv 0 \pmod{k - 1}$ and $k -1 \equiv k^2 - 1 \equiv -1
\pmod{k}$, $m$ is also restricted by the existence of these two sets of
group divisible designs by Lemma~\ref{lem:ipbd:gdd-small}.  We may also choose
$m$ so it is of the order $1 / \epsilon$.

Let $z = w \mod{k}$.  By Theorem~\ref{thm:ipbd:K:fixed-w}, there exist
IPBD$((u (k - 1) + z; z), K)$ for all admissible $u \ge u_0(z, K)$.  As $z$
has only $k$ possible congruence classes, we can define $u_0(k) :=
\max\{u_0(z, K)\}$ to be independent of $z$.  Let $y = w - z$; then $y
\equiv 0 \pmod{k (k - 1)}$.

We construct the incomplete pairwise balanced designs starting from a
transversal design.  From \cite{CES}, there exist TD$(m + 2, n)$ for all $n
\ge n_0(m)$.  Then, for $v - w$ sufficiently large, we can express $v - w =
(k - 1) (mn + p)$ such that $k \mid n$, $n \ge n_0(m)$, and both $n, p \ge
u_0(K)$.  We delete all but $p$ points of one of the groups of the
transversal design to obtain a GDD$(n^m p^1 n^1, \{m + 1, m + 2\})$, where
the last group of $n$ is separated for convenience of notation.  We now
assign weights to the points of the group divisible design.  Each point in
the first $m + 1$ groups receives a weight of $k - 1$ and each point in the
final group receives a weight in $R$.   Hence, our ingredient group
divisible designs are of the form GDD$((k - 1)^m x^1, k)$ and GDD$((k - 1)^{m +
1} x^1, k)$, whose existence was shown above, so the result of applying
Wilson's Fundamental Construction is a GDD$(((k - 1) n)^m ((k - 1) p)^1 t^1,
k)$, where $t \in n * R$, the set of $n$-fold sums of integers taken from
$R$.  Finally, since there exists an IPBD$(((k - 1) n + z; z), K)$ and an
IPBD$(((k - 1) p + z; z), K)$, there also exists an IPBD$(((k - 1) (nm + p)
+ t + z; t + z), K)$ by Construction~\ref{const:gdd-ipbd}.

It remains to consider the values of $n * R$.  We need each possible
hole size, so we wish to find an arithmetic progression having difference $k
(k - 1)$, which is precisely the difference between the two smaller terms of
$R$.  If moving to the next value in the arithmetic progression
requires introducing an additional $r$ term compared to the previous sum,
then some number, say $c$, terms of $k^2 - 1$ must be removed, and $c - 1$
terms of $k - 1$ must also be introduced. It is an easy calculation that
\begin{equation}
\label{estimate-c}
c = \frac{m - 1}{k (k - 2)} - \frac{k + 1}{k} < \frac{n (k - 2)}{m - 1},
\quad \mbox{for all sufficiently large $n$.}
\end{equation}
If we let $t_{\max}$ be the largest value of the arithmetic progression,
then 
\begin{align*}
t_{\max} &= (n - (c - 1)) \frac{(k - 1)(m - 1)}{k - 2} + (c - 1) (k^2 - 1)
\\
&\ge (k - 1) n \left( \frac{m - 1}{k - 2} - 1 \right), \quad \mbox{by $(\ref{estimate-c})$.}
\end{align*}
It follows that we achieve constructions for point-hole ratios as small as
$$ \frac{v}{w} < \frac{(k - 1) n (m + 1)}{t_{\max}} + 1 < (1 + O(1/m)) (k -
2) + 1 < k - 1 + \epsilon $$
as required.
\end{proof}

To finish our proof, we combine individual applications of the preceding result for each $k \in K_0$.  But, for this, we
first must prove the following technical lemma to apply each step.

\begin{lemma} \label{lem:ipbd:tech}
Given $K \subseteq \Z_{\ge 2}$, $K_0= \{k_1, \ldots, k_n\}$ with $\alpha(K_0)=\alpha(K)$, and admissible parameters $(v; w)$ for incomplete pairwise
balanced designs with block sizes in $K$, then for all sufficiently large
$v$, we can write $v - w = \sum_{k \in K_0} c_k (k - 1)$, for nonnegative
integers $c_k$ in
such a way that the parameter pair $(\sum_{j = 1}^i c_{k_j} (k_j - 1) + w; \sum_{j = 1}^{i - 1} c_{k_j} (k_j - 1) + w)$ is also
admissible for each $i = 1, \ldots, n$.
\end{lemma}

\begin{proof}
Let $K_l = \{k_1, \ldots, k_l\}$, $l = 1, \ldots, n$, and let $a_l =
\gcd\{k_i - 1: i = 1, \ldots, m\}$.  We assume by induction on $l$ that if
$(c + w; w)$ are admissible parameters for incomplete pairwise balanced
designs with block sizes in $K$ and $c \equiv 0 \pmod{a_l}$, then for all
sufficiently large $c$, we can write $c = \sum_{i = 1}^l c_{k_i} (k_i - 1)$
in such a way that $(\sum_{j = 1}^i c_{k_j} (k_j - 1) + w; \sum_{j = 1}^{i -
1} c_{k_j} (k_j - 1) + w)$ is also admissible for each $i = 1, \ldots, l$.

The case $l = 1$ is trivial as $a_1 = k_1 - 1 \mid c$ and so we can choose $c_1 =
\frac{c}{k_1 - 1}$.  Note that $(c + w; w)$ is admissible by assumption.
Now, assume the result holds for all $l'$, $1 \le l'<l$.  We show the result also holds
for $l$.  Let $b = k_l - 1$.  Since $a_l = \gcd(a_{l - 1}, b)$ and $c
\equiv 0 \pmod{a_l}$ by assumption, then $c = a_{l - 1}x + by$ has integer
solutions in $x$ and $y$.  If $x_0, y_0$ is a particular solution, then
every solution is of the form $(x, y) = (x_0 + nb, y_0 - na_{l - 1})$ for $n
\in \mathbb{Z}$.  It remains to find a solution $x, y$ such that $(a_{l -
1}x + w; w)$ is admissible, that is, where $a_{l - 1}x (a_{l-1}x + 2w - 1)
\equiv 0 \pmod{\beta(K)}$.  If $a_{l-1}$ and $\beta(K)$ have a common
factor, it can be divided out.  Put $\beta' :=\beta(K)/\gcd(a_{l-1}, \beta(K))$.

{\sc Claim}. $\gcd(a_{l-1}, \beta') =1$.  

To verify this, we proceed by contradiction, and assume there
exists a prime $p$ such that $p \mid a_{l - 1}$ and $p \mid \beta'$.
Suppose $p^e \exdiv a_{l - 1}$.  Then it must be the case that $p^{e + 1}
\mid \beta(K)$.  As $p \mid a_{l - 1}$, it follows that $p \mid k_i - 1$ and so $\gcd(p, k_i) = 1$ for each $i = 1,
\ldots, l - 1$.  Now, since $p^{e + 1} \mid \beta(K)$, we have
$p^{e + 1} \mid k_i - 1$ for each $i$.  From this, $p^{e + 1} \mid
a_{l - 1}$, which is the contradiction proving the claim. 

It remains to find a solution such that
$a_{l - 1}x (a_{l - 1}x + 2w - 1) \equiv 0 \pmod{\beta'}$.  As $a_{l - 1}x =
c - by$, we have
\begin{align}
\label{tech-cong}
\nonumber
a_{l- 1}x (a_{l- 1}x + 2w - 1) &\equiv (c - by) (c + 2w - by - 1) \\
&\equiv -by (2c + 2w - 1 + y) \pmod{\beta'}.
\end{align}
Choosing $y \equiv 1 - 2c - 2w
\pmod{\beta'}$, which is possible by the claim, ensures by (\ref{tech-cong}) that $(a_{l - 1}x + w; w)$ is admissible.  Then, by 
induction hypothesis, we can write $a_{l - 1} = \sum_{i = 1}^{l - 1} c_{k_i}
(k_i - 1)$ in such a way that 
$$\left(\sum_{j = 1}^i c_{k_j} (k_j - 1) + w;\sum_{j = 1}^{i - 1} c_{k_j} (k_j - 1) + w \right)$$ 
is also admissible for each $i<l$, and the result follows.
\end{proof}

Our asymptotic result on incomplete pairwise balanced designs with multiple
block sizes now easily follows from the previous two results.

\begin{proof}[Proof of Theorem~\ref{ipbd}]
Let $K_0=\{k_1,\dots,k_n\}$.  For large $v - w$, we can write $$v - w = \sum_{k\in K_0} c_k (k - 1)$$ in
such a way satisfying the conclusion of Lemma~\ref{lem:ipbd:tech}, and
also so that, by Lemma~\ref{lem:ipbd:recur}, there exists an IPBD$((\sum_{j
= 1}^i c_{k_j} (k_j - 1) + w; \sum_{j = 1}^{i - 1} c_{k_j} (k_j - 1) + w),
K)$ for $i = 1, \ldots, n$.  The required design exists by inductively
filling holes.
\end{proof}

\section{Proof of the main result and discussion}

We return to our intended application to IMOLS.  First, a latin square $L$
on symbols $[n]$ is \emph{idempotent} if $L_{ii}=i$ for each $i=1,\dots,n$.
It is well known that in a plane ($q-1$ MOLS) of order $q$, we can choose
all but one of these squares to be idempotent.

\begin{prop}[See \cite{Handbook}]
\label{idempotent}
For prime powers $q$, there exist $q-2$ mutually orthogonal idempotent latin
squares of order $q$.
\end{prop}

Idempotent latin squares enjoy the feature that they can be `glued' along the
blocks of a pairwise balanced design; see \cite{CD}, for instance.
To illustrate this construction, consider the case $n=7$, arising from the Fano plane having blocks
$$\{1,2,3\},\{1,4,5\},\{1,6,7\},\{2,4,6\},\{2,5,7\},\{3,4,7\},\{3,5,6\}.$$
We get the idempotent square shown below, with the template square on the first block highlighted.
$$\begin{array}{|ccccccc|}
\hline
\bf1&\bf 3& \bf 2&5&4&7&6 \\
\bf 3&\bf 2& \bf 1&6&7&4&5 \\
\bf 2&\bf 1& \bf 3&7&6&5&4 \\
5&6&7&4&1&2&3  \\
4&7&6&1&5&3&2 \\
7&4&5&2&3&6&1 \\
6&5&4&3&2&1&7 \\
\hline
\end{array}
$$

More generally, we have the following construction for IMOLS from an IPBD and MOLS.  The key idea is that every block intersects the hole in at most one point, and in that case we use template MOLS with a $1\times 1$ subsquare removed.  The above square with bold entries removed is an example.

\begin{lemma}  \label{const:imols}
Suppose there exists an IPBD$((n;m), K)$ and, for each $k \in K$, there
exist $t$ idempotent MOLS of order $k$.  Then there exist $t$-IMOLS$(n;m)$.
\end{lemma}

Now, Lemma~\ref{const:imols} and Theorem~\ref{ipbd} (with a careful choice
of $K$) are enough to do the job.

\begin{proof}[Proof of Theorem~\ref{main}]
We show that if $2^f$ is the smallest power of 2 greater than $t + 1$, then
there exist $t$-IMOLS$(n;m)$ for all sufficiently large $n,m$ such that $n
\ge 2^{2f+1} m$.  Let $K_0=\{2^f,2^{f+1}\}$ and $K=K_0 \cup \{3^{2f+1}\}$.
Observe that $\alpha(K_0)=\gcd\{2^f-1,2^{f+1}-1\}=1$ and $\beta(K)=2$.  It
follows by Theorem~\ref{ipbd} that there exist IPBD$((n;m), K)$ for all
sufficiently large integers $n,m$ satisfying $n \ge 2^{2f+1} m$.  Moreover,
there exist $t$ mutually orthogonal idempotent latin squares of each order
$k \in K$ by Proposition~\ref{idempotent}.  It follows by
Lemma~\ref{const:imols} that there exist $t$-IMOLS$(n;m)$.
\end{proof}

We conclude with a few remarks.  First, as is seen from the proof, special
values of $t$ permit a weaker hypothesis than our $n \ge 8(t+1)^2 m$.  But
we do not get a qualitative improvement where the quadratic in $t$ is
reduced to something linear.

There are, though, some special cases in which we can do better.  When $n$
and $m$ fall into certain congruence classes with respect to $t$, we can get
away with a singleton set $K_0$ and hence obtain constructions near a ratio
that is linear in $t$.  Indeed if $p$ is the next prime following $t$, we
have $t$ MOLS of order $p$.  It follows that, after a product, we have
$t$-IMOLS$(n;m)$ with $n=pm$ for all large $m$.  The construction in
\cite[Proposition 3.4]{BvR} has the versatility to `multiply and shift', but
it requires $t$ MOLS of order near the ratio $n/m$.  It is unclear if some
variant of a product construction can, in the general case, get close to the
hole size bound.

On the other hand, if the following strengthening of Theorem~\ref{ipbd}
could be proved, our ratio $8(t+1)^2$ would become linear in $t$.

\begin{conj}
\label{conj}
Let $K \subseteq \Z_{\ge 2}$.  For any real $\epsilon > 0$, there exist
IPBD$((v; w), K)$ for all sufficiently large $v, w$ satisfying
$(\ref{local})$, $(\ref{global})$, and $v >  (\min K - 1 + \epsilon) w$.
\end{conj}

It is likely that settling Conjecture \ref{conj} will require new
techniques, perhaps including resolvable designs with mixed block sizes.

Finally, it is noteworthy that our constructed IMOLS are `approximately symmetric'.

\end{document}